\documentclass{mathart}

\title{On the hypergraph connectivity of skeleta of polytopes}
\author{Daniel Hathcock}
\address{Department of Mathematical Sciences, Carnegie Mellon University, Pittsburgh PA, US}\email{dhathcoc@andrew.cmu.edu}

\author{Josephine Yu}
\address{School of Mathematics, Georgia Institute of Technology, Atlanta GA, US}\email{jyu@math.gatech.edu}


\let\dual\Delta

\begin{document}

\begin{abstract}We show that for every $d$-dimensional polytope, the hypergraph whose nodes are $k$-faces and whose hyperedges are $(k+1)$-faces of the polytope is strongly $(d-k)$-vertex connected, for each $0 \leq k \leq d- 1$.
\end{abstract}

\maketitle

\section{Introduction}
\label{sec:intro}

Balinski proved that the edge graph of any $d$-dimensional polytope is $d$-vertex connected~\cite{balinski1961graph}. That is, removing fewer than $d$ of the vertices leaves the remaining vertices connected via edges. 
A number of natural generalizations of this result have since been
investigated. Sallee found bounds for several different notions of
connectivity of incidence graphs between $r$-faces and $s$-faces of a
polytope \cite{sallee1967incidence}. More recently, Athanasiadis
considered the graphs $\mathcal{G}_k(P)$ for a convex polytope $P$, whose nodes are the $k$-faces of $P$, and with two nodes adjacent if the corresponding $k$-faces are both contained in the same $(k+1)$-face. Vertex connectivity of $\mathcal{G}_k(P)$ is equivalent to one of the connectivity notions on the incidence graphs considered by Sallee. Athanasiadis described exactly the minimum vertex connectivity of $\mathcal{G}_k(P)$ over all $d$-polytopes for every $k$ and $d$ \cite{athanasiadis2009graph}.

Let $P$ be a convex $d$-dimensional polytope.  We denote by $\mathcal{H}_k(P)$ the hypergraph whose nodes are the $k$-faces of polytope $P$, and whose hyperedges correspond naturally to the $(k + 1)$-faces of $P$. We say a hypergraph is {\bf strongly $\alpha$-vertex connected} if removing fewer than $\alpha$ nodes along with all hyperedges incident to each removed node leaves the remaining nodes connected. Using tropical geometry, Maclagan and the second author showed that for every \textit{rational} $d$-polytope, $\mathcal{H}_k(P)$ is strongly $(d - k)$-vertex connected \cite{maclagan2019higher}. Our main result is generalizing this statement to all polytopes: 

\begin{theorem}
\label{thm:mainthm}
For every $d$-polytope $P$, the hypergraph $\mathcal{H}_k(P)$ is strongly $(d-k)$-vertex connected, for each $0 \leq k \leq d- 1$. 
\end{theorem}
The result is tight.  For simple polytopes, each $k$-face is contained in exactly $d-k$ of the $(k+1)$-faces, so the hypergraph $\mathcal{H}_k(P)$  cannot have higher connectivity.

\section{Proof of the result}
\label{sec:proofs}

We say that a pure $k$-dimensional polyhedral complex is {\bf $c$-connected through codimension one} if after removing fewer than $c$ {\em closed} maximal faces, the remaining maximal faces are connected via paths through faces of dimension $k-1$. That is, for any two remaining maximal faces $F, F'$, there remains a sequence $F = G_1, \ldots, G_\ell = F'$ of maximal faces such that for each $i$, $G_i \cap G_{i + 1}$ is a face of dimension $k-1$ not belonging to a removed face. The {\bf $m$-skeleton} of a polytope $Q$ is the polyhedral complex whose maximal faces are the $m$-dimensional faces of $Q$. Then  Theorem~\ref{thm:mainthm}  can be rephrased as the following equivalent form on the  polar dual $Q = P^{\dual}$.  

\begin{theorem}
\label{thm:dual}
For every $d$-polytope $Q$, the $(d-k-1)$-skeleton is $(d-k)$-connected through codimension one, for each $0 \leq k \leq d- 1$. Equivalently, the $k$-skeleton of $Q$ is $(k + 1)$-connected through codimension one for each $0 \leq k \leq d-1$. 
\end{theorem}

We will need some lemmas before proceeding with the proof by induction on dimension.

\begin{lemma}
\label{lemma:hyperplane}
Let $F, G, R$ be three distinct $k$-faces of a $d$-polytope $Q$, for some $1 \leq k \leq d-1$. Then there is a hyperplane intersecting $F$ and $G$ and avoiding $R$. Moreover, the hyperplane can be chosen to avoid all vertices of $Q$. 
\end{lemma}

\begin{proof}
Let $f \in F$ and $g \in G$ be relative interior points, and let $L$ be the line through $f$ and $g$.  Let $Q'$ be the smallest face of $Q$ containing $F \cup G$.
By convexity, $L \cap Q \subset Q'$ and $L$ meets the boundary of $Q'$ only at the two points $f$ and $g$.  In particular $L$ does not meet $R$ or any other face of dimension $\leq k$.

We may assume that $Q$ is a $d$-dimensional polytope in
$\mathbb{R}^d$.  Let $\pi$ be a corank one linear map from
$\mathbb{R}^d$ to $\mathbb{R}^{d-1}$ such that the image of $L$ is a point. Then the image $R' = \pi(R)$ does not contain $\pi(L)$, and each vertex $v_1, \ldots, v_n$ of $Q$ has $v_i' = \pi(v_i) \neq \pi(L)$ since $L$ does not contain any of the vertices.

Since $R'$ is convex and does not contain $\pi(L)$, there is a hyperplane through $\pi(L)$ which does not meet $R'$.   Since $R'$ is compact, the set of normal vectors of such hyperplanes form a full dimensional open set in  $\mathbb{RP}^{d-1}$.  (More precisely, it is the interior of the dual cone, and its negative, of the pointed cone generated by $R'$ after a translation that sends $\pi(L)$ to the origin.)  On the other hand, the condition that such a hyperplane contains each $v_i'$ is a codimension one closed condition. Thus, as there are finitely many $v_i'$, the cone of such normal vectors restricted to those whose hyperplane {\em does not} contain any $v_i'$ is non-empty. In particular, there is a hyperplane $H'$ through  $\pi(L)$ which does not meet $R'$ or any of the $v_i'$.  Its preimage $\pi^{-1}(H)$ is a desired hyperplane.
\end{proof}

\begin{lemma}
\label{lemma:hyperplaneIntersect}
Let $Q$ be a polytope and $H$ a hyperplane intersecting $Q$ but not containing any vertices of $Q$.  
The map $\phi : F \mapsto F \cap H$ is a poset isomorphism
from the poset of faces of $Q$ that meet $H$ to the face poset of $Q \cap H$. 
\end{lemma}

\begin{proof}
For any face $F$ of $Q$ which meets $H$, since $H$ does not contain any vertices of $F$, $F$ is not contained in $H$ and $H$ meets the relative interior of $F$,  so $\dim(F \cap H) = \dim(F) -1$. 
Moreover, $F \cap H$ is indeed a face of $Q \cap H$:  any supporting hyperplane for $F$ in $Q$ is also a supporting hyperplane for $F \cap H$ in $Q \cap H$.  On the other hand, for any face $F'$ of $Q \cap H$, let $x \in F'$ be a relative interior point in $F'$, and let $F$ be the unique face of $Q$ for which $x$ is a relative interior point.  Then $x$ is also in the relative interior of $F \cap H$.  Since $F'$ and $F\cap H$ are two faces of $Q \cap H$ that meet in their relative interiors, we have $F \cap H = F'$.
So $\phi$ is a surjective map between the desired sets.  
If $F \cap H = G \cap H$ for $k$-faces $F, G$ meeting $H$, then $F$ and $G$ would have a common relative interior point, which implies $F=G$.  Thus $\phi$ is injective.  It is clear that $\phi$ preserves the inclusion relation.
\end{proof}

\begin{proof}[Proof of Theorem \ref{thm:dual}] 

We will use induction on $k$.   The statement is trivial for $k = 0$, as we are not removing any faces, and the vertices of a polytope are connected through the empty face.  The case when $k = 1$ is clear, as removing a single edge does not disconnect the vertex-edge graph of any polytope.

Suppose $2 \leq k \leq d-1$.  Let $Q$ be a $d$-polytope and $\mathcal{B}$ be any set of $k$ $k$-faces of $Q$ to remove. We need to find a path between any two $k$-faces $F, G \not \in \mathcal{B}$, through codimension-one faces, which we will call {\em ridge paths}. 
Arbitrarily choose any $R \in \mathcal{B}$. Lemma \ref{lemma:hyperplane} gives a hyperplane $H$ intersecting $F$ and $G$, and avoiding $R$ and vertices of $Q$. Let $Q' = Q \cap H$. Since $H$ intersects $F$ and $G$, $F' = F \cap H$ and $G' = G \cap H$ are two $(k-1)$-faces of $Q'$ by Lemma~\ref{lemma:hyperplaneIntersect}. Moreover, each face in $\mathcal{B} \setminus \{R\}$ corresponds to at most one $(k-1)$-dimensional face in $Q'$. Call these faces $\mathcal{B}'$.
As $\abs{\mathcal{B}'} \leq k-1$, by induction there is a ridge path in $Q'$ connecting $F'$ to $G'$ and avoiding each face in $\mathcal{B}'$. Using Lemma \ref{lemma:hyperplaneIntersect}, we can lift this path back up to a ridge path connecting $F$ to $G$ in $Q$ avoiding $\mathcal{B}$.
\end{proof}

\noindent{\bf Acknowledgements}\\
We thank Diane Maclagan for discussions and the referee for comments which helped improve the exposition. JY was partially supported by NSF-DMS grant \#1855726.

\bibliographystyle{alpha}
\bibliography{bibliography}

\end{document}